\documentclass[11pt]{amsart}

\usepackage{amsmath}
\usepackage{amssymb}
\usepackage{latexsym}

\newtheorem{theorem}{Theorem}
\newtheorem{lemma}[theorem]{Lemma}
\theoremstyle{remark}
\newtheorem{remark}[theorem]{Remark}

\newcommand{\RR}{\mathbb{R}}
\newcommand{\ZZ}{\mathbb{Z}}
\newcommand{\NN}{\mathbb{N}}
\newcommand{\TT}{\mathbb{T}}
\newcommand{\CC}{\mathbb{C}}
\newcommand{\II}{\mathbb{I}}
\newcommand{\pa}{\partial}
\newcommand{\cT}{\mathcal{T}}
\newcommand{\cL}{\mathcal{L}}
\newcommand{\cF}{\mathcal{F}}
\newcommand{\cX}{\mathcal{X}}
\newcommand{\bu}{\mathbf{u}}
\newcommand{\buzero}{\mathbf{u}_0}
\newcommand{\bv}{\mathbf{v}}
\newcommand{\bk}{\mathbf{k}}
\newcommand{\tbk}{\widetilde{\mathbf{k}}}
\newcommand{\bx}{\mathbf{x}}
\newcommand{\by}{\mathbf{y}}
\newcommand{\un}{\mathbf{u}^{(n)}}
\newcommand{\vn}{\mathbf{v}^{(n)}}
\newcommand{\Un}{\mathbf{U}^{(n)}}
\newcommand{\Vn}{\mathbf{V}^{(n)}}
\newcommand{\bU}{\mathbf{U}}
\newcommand{\bV}{\mathbf{V}}
\newcommand{\balpha}{\boldsymbol{\alpha}}
\newcommand{\rL}{\mathring{L}}

\newcommand{\what}{\widehat}

\usepackage{verbatim}
\usepackage{esint}

\newtheorem{proposition}[theorem]{Proposition}

\begin{document}

\title[The 2D Kuramoto-Sivashinsky Equation]{Global existence and analyticity
for the 2D Kuramoto-Sivashinsky Equation}
\author{David M. Ambrose}
\address{Department of Mathematics, Drexel University, Philadelphia, PA 19104,
USA}
\email{dma68@drexel.edu}
\author{Anna L. Mazzucato}
\address{Department of Mathematics, Penn State University, University Park, PA
16802, USA}
\email{alm24@psu.edu}

\dedicatory{In loving memory of George R. Sell.}

\begin{abstract} There is little analytical theory for the behavior of solutions of the Kuramoto-Sivashinsky equation in two
spatial dimensions over long times.  We study the case in which the spatial domain is a two-dimensional torus.
In this case, the linearized behavior depends on the size of the torus -- in particular, for different sizes of the domain, there
are different numbers of linearly growing modes.  We prove that small solutions exist for all time if there are no linearly
growing modes, proving also in this case that the radius of analyticity of solutions grows linearly in time.
In the general case (i.e., in the presence of a finite number of
growing modes), we make estimates for how the radius of analyticity of solutions changes in time.
\end{abstract}



\maketitle

\section{Introduction}

In $n$ spatial dimensions, the Kuramoto-Sivashinsky equation is
\begin{equation}\label{phiEquation}
\phi_{t}+\frac{1}{2}|\nabla\phi|^{2}=-\Delta^{2}\phi-\Delta\phi.
\end{equation}
In the case $n=1,$ we introduce $u=\phi_{x}$ and differentiate \eqref{phiEquation}, finding the equation
\begin{equation}\label{1DKS}
u_{t}+uu_{x}=-\Delta^{2}u-\Delta u.
\end{equation}
The differentiated form \eqref{1DKS} is also referred to as the
Kuramoto-Sivashinsky equation.
Solutions of the initial value problem for \eqref{1DKS} have been shown to exist
for all times \cite{tadmor}, and
stability of $u=0$ has also been demonstrated \cite{goodman}, \cite{NST}.

A fundamental difficulty in the study of the Kuramoto-Sivashinsky equation is the lack of a maximum principle,
because of the presence of the fourth derivative term on the right-hand side of \eqref{phiEquation} or \eqref{1DKS}.
In the one-dimensional case, however, \eqref{1DKS} allows conservation of the $L^{2}$ norm, and this has proved quite
useful in previous studies.  In two spatial dimension, this conservation property is not present, and much less progress
has been made.

We may state the 2D Kuramoto-Sivashinksy equation in its differentiated form, letting $u=\phi_{x}$ and $v=\phi_{y}.$
This leads to the following system, which we call 2DKS in the sequel:
\begin{equation}\label{2DKSu}
u_{t}+uu_{x}+vv_{x}=-\Delta^{2}u-\Delta u,
\end{equation}
\begin{equation}\label{2DKSv}
v_{t}+uu_{y}+vv_{y}=-\Delta^{2}v-\Delta v,
\end{equation}
\begin{equation}
u_{y}=v_{x}.
\end{equation}
We supplement this with initial conditions:
\begin{equation}\label{IC}
u(\cdot,0)=u_{0},\qquad v(\cdot,0)=v_{0}.
\end{equation}
It is easy to see that if the initial data is a gradient, then so is the
solution, at least for strong solutions, but this condition is not always needed in
the analysis.
We will sometimes write the equations for $(u,v)$ as above in \eqref{2DKSu}, \eqref{2DKSv}, but we may also
write them in a vector form.  We let $\mathbf{u}=(u,v),$ and then we have the following equivalent version of the
evolution equations:
\begin{equation}\nonumber
\mathbf{u}_{t}+\mathbf{u}\cdot\nabla\mathbf{u}=-\Delta^{2}\mathbf{u}
-\Delta\mathbf{u},
\end{equation}
where we have used the fact that $\bu$ is a curl free to write the
nonlinearity as $\, (\nabla |\bu|^2)/2 = \bu \cdot \nabla \bu$.

In two dimensions, Molinet considered a modification of the
Kuramoto-Sivashinsky equation, known as the
Burgers-Sivashinsky model (Burgers-Sivashinsky was also considered by Goodman \cite{goodman}).
Unlike Kuramoto-Sivashinsky, the Burgers-Sivashinsky model admits a maximum principle, and Molinet leverages this
to find global existence of small solutions \cite{molinet}.  For 2DKS, Sell and Taboada proved global existence of solutions
in the case of a thin domain \cite{sellTaboada}; more recent work on 2DKS in thin domains is \cite{newThinResults},
\cite{molinet2}.

In the present work, we provide global existence theorems for small data in the
case of a domain that is not thin, but that satisfies a size condition (the
periods must be less than $2\pi$); as long as the domain admits no growing
mode for the linear part of the evolution, we prove that for initial data
small either in the Wiener algebra or in $L^{2},$ the solution exists for all
time.  For the result in the Wiener algebra,
an automatic consequence is that solutions are analytic at all positive times, with the radius of analyticity growing linearly
in time.    In the one-dimensional case, significant work has been done on tracking how the behavior of solutions depends
upon the size of the domain, such as by Giacomelli and Otto \cite{giacomelliOtto}; we do not provide such a detailed
description of dependence on domain size, but instead only draw the distinction as to whether linearly growing
modes are present or not.

In the general case (not restricting the size of the domain, and thus allowing linearly growing modes to be present),
we again prove results both for data in the Wiener algebra and in $L^{2}.$  For general $L^{2}$ data (not necessarily small)
in such a general domain, we prove that solutions become analytic at positive
times, as long as they exists, and provide a lower bound
on the size of the radius of analyticity.  Our bound initially grows like
$t^{1/4},$ and then decays exponentially on the interval of time where the
solution exists. In this paper, we do not directly exploit a Fourier
representation of the solution, as in the seminal work of Foias and
Temam on the Navier-Stokes equation \cite{FT89}. Rather, we adapt the approach
of Gruji\'c and Kukavica \cite{GK98}, which is
based on a suitable regularization of the equations, bounds on mild and strong
solutions, and a version of Montel's Theorem for normal families in several
complex variables. The advantage of this approach is that it is directly
applicable to treat analyticity with data in $L^p$, $p\ne 2$, though we do not
pursue existence in $L^p$ in this work.

There are several works in the literature using the Wiener algebra to explore
analyticity of solutions to non-linear PDEs, especially in the context of fluid
mechanics. Indeed, the radius of analyticity can be linked to the decay of the
power spectrum for the solution \cite{BJMT14}, which in turn gives rigorous
bounds on the turbulent dissipation scale \cite{DT95} The Wiener algebra
approach is especially useful to derive estimates on the growth of Gevrey norms
(we refer in particular to the work of Oliver and Titi \cite{OT00,OT01}).

The plan of the paper is as follows: in Section 2, we prove existence of small solutions for all time with data in the Wiener
algebra in the case that there are no linearly growing modes.  In Section 3, we treat the general case for data in the Wiener
algebra, proving short-time existence.  These results in the Wiener algebra automatically provide a lower bound on the
radius of analyticity of solutions.  In Section 4 we prove a short-time existence theorem for initial data in $L^{2},$ with
or without growing modes and for any size data.  In Section 5 we further establish that without growing modes and for small
$L^{2}$ data, these solutions exist for all time.  In Section 6, returning to the general case with $L^{2}$ data, we establish
a lower bound for the radius of analyticity.

\subsection*{Acknowledgments}
The authors thank Edriss Titi for helpful conversations.  The authors are also
grateful to the National Science Foundation
for support through NSF grants DMS-1515849 (to Ambrose) and DMS-1615457
(to Mazzucato).
The authors acknowledges the hospitality and support of the Institute for
Computational and Experimental Research in Mathematics (ICERM) during the
Semester Program on "Singularities and Waves In Incompressible Fluids", where
part of this work was discussed. ICERM receives major funding from NSF and Brown
University.

\section{No growing modes: A global existence theorem}  \label{s:wiener}

We prove a global existence theorem for small solutions of 2DKS; this uses function spaces based on the Wiener algebra,
and is inspired by the proof of small vortex sheets for all time by Duchon and Robert \cite{duchonRobert}.  We note that
the first author has subsequently used these same and related ideas for various problems in the papers
\cite{ambroseMFG1}, \cite{ambroseMFG2}, and \cite{milgromAmbrose}.

Making a straightforward computation, we arrive at the Duhamel representation of the solution
to \eqref{2DKSu}, \eqref{2DKSv}, \eqref{IC}:
\begin{equation}\label{duhamel1}
u(\cdot,t)=e^{-t(\Delta^{2}+\Delta)}u_{0}
-\int_{0}^{t}e^{-(t-s)(\Delta^{2}+\Delta)}(uu_{x}+vv_{x})(\cdot,s)\ ds,
\end{equation}
\begin{equation}\label{duhamel2}
v(\cdot,t)=e^{-t(\Delta^{2}+\Delta)}v_{0}
-\int_{0}^{t}e^{-(t-s)(\Delta^{2}+\Delta)}(uu_{y}+vv_{y})(\cdot,s)\ ds.
\end{equation}
We thus introduce an operator, $\mathcal{T},$ given by the right-hand sides of these equations:
\begin{equation}\nonumber
\mathcal{T}(u,v)=\left(\begin{array}{c}
e^{-t(\Delta^{2}+\Delta)}u_{0}
-\int_{0}^{t}e^{-(t-s)(\Delta^{2}+\Delta)}(uu_{x}+vv_{x})(\cdot,s)\ ds\\
e^{-t(\Delta^{2}+\Delta)}v_{0}
-\int_{0}^{t}e^{-(t-s)(\Delta^{2}+\Delta)}(uu_{y}+vv_{y})(\cdot,s)\ ds
\end{array}\right).
\end{equation}
We will look for solutions of 2DKS by finding fixed points of the operator $\mathcal{T}.$

\subsection{Function spaces}

We consider the torus $\mathbb{T}^{2}$ with dimensions $[0,L_{1}]\times[0,L_{2}].$
On such a torus, a function $f$ can be expressed in terms of its Fourier series as follows:
\begin{equation}\nonumber
f(x)=\sum_{k\in\mathbb{Z}^{2}}\hat{f}(k)\exp\left\{2\pi i\left(
\frac{k_{1}x}{L_{1}}+\frac{k_{2}y}{L_{2}}\right)\right\}.
\end{equation}
Here, we have denote $k=(k_{1},k_{2})\in\mathbb{Z}^{2}.$  We have denoted the Fourier coefficients of $f$
as $\hat{f}(k),$ but in the sequel we may also denote them as $\mathcal{F}(f)(k).$
On this torus, we have the following symbols for the operators $\partial_{x}$ and $\partial_{y}:$
\begin{equation}\nonumber
\left(\mathcal{F}\partial_{x}\right)(k)=\frac{2\pi i k_{1}}{L_{1}},
\qquad
\left(\mathcal{F}\partial_{y}\right)(k)=\frac{2\pi i k_{2}}{L_{2}}.
\end{equation}
We mention also that the symbol of $\Delta^{2}+\Delta$ is given by
\begin{equation}\label{symbolOfLinearPart}
\left(\mathcal{F}(\Delta^{2}+\Delta)\right)(k)
=
\frac{16\pi^{4}k_{1}^{4}}{L_{1}^{4}}+\frac{16\pi^{4}k_{1}^{2}k_{2}^{2}}{L_{1}^{2}L_{2}^{2}}
+\frac{16\pi^{4}k_{2}^{4}}{L_{2}^{4}}-\frac{4\pi^{2}k_{1}^{2}}{L_{1}^{2}}-\frac{4\pi^{2}k_{2}^{2}}{L_{2}^{2}}.
\end{equation}
We introduce some notation: we let $\sigma(k)$ denote the right-hand side of \eqref{symbolOfLinearPart}.
In the current section, we are studying the case in which the linear operator of the Kuramoto-Sivashinsky equation,
which is $-\Delta^{2}-\Delta,$ yields no growing modes.  We see from \eqref{symbolOfLinearPart}
that this means we set the conditions $L_{1}\in(0,2\pi)$ and $L_{2}\in(0,2\pi).$  With these conditions satisfied,
we have $\left(\mathcal{F}(\Delta^{2}+\Delta)\right)(k)>0,$ for all $k\in\mathbb{Z}^{2}\setminus\{0\}.$

We introduce function spaces based on the Wiener algebra.
For any $\rho\geq0,$ we define $B_{\rho}$ to be a set of functions from $\mathbb{T}^{2}$
to $\mathbb{R}$ as follows:
\begin{equation}\nonumber
B_{\rho}=\left\{f: |f|_{\rho}<\infty\right\},
\end{equation}
with the norm defined by
\begin{equation}\nonumber
|f|_{\rho}=\sum_{k\in\mathbb{Z}^{2}}e^{\rho|k|}|\hat{f}(k)|.
\end{equation}
Note that if $\rho=0,$ then we are requiring $\hat{f}\in \ell^{1},$ and so the space $B_{0}$ is
exactly the Wiener algebra.

We next define a version of these spaces for functions which also depend on time.
Let $\alpha>0;$ define $\mathcal{B}_{\alpha}$ to be the set of functions continuous from $[0,\infty)$ to
$B_{0},$ such that for all such
$f:\mathbb{R}\times[0,\infty)\rightarrow\mathbb{R}$ we have
\begin{equation}\nonumber
\mathcal{B}_{\alpha}=\left\{f:\|f\|_{\alpha}<\infty\right\},
\end{equation}
with the norm defined by
\begin{equation}\nonumber
\|f\|_{\alpha}=\sum_{k\in\mathbb{Z}^{2}}\sup_{t\in[0,\infty)}\left(e^{\alpha t|k|}|\hat{f}(k,t)|\right)
\end{equation}

Notice that for any $\alpha>0,$ if $f\in\mathcal{B}_{\alpha},$ then for all $t>0,$ we have that
$f(\cdot,t)$ is analytic, but $f(\cdot,0)$ need not be analytic (it is only in the Wiener algebra).
Furthermore, it is well-known that the Wiener algebra is a Banach algebra (hence the name), and
the spaces $\mathcal{B}_{\alpha}$ inherit this property.  To see the algebra property, first note that
$B_{\rho}$ is a Banach algebra, and this can be seen from the following:
\begin{equation}\nonumber
\left| e^{\rho|k|}\widehat{fg}(k)\right|\leq\sum_{j}\left|e^{\rho |k-j|}\hat{f}(k-j)\right|
\left| e^{\rho |j|}\hat{g}(j)\right|.
\end{equation}
Summing in $k,$ we see that $|fg|_{\rho}\leq|f|_{\rho}|g|_{\rho}.$
The same considerations, and some elementary manipulations of the supremum, imply that
if $f\in\mathcal{B}_{\alpha}$ and $g\in\mathcal{B}_{\alpha},$ we have
\begin{equation}\nonumber
\|fg\|_{\alpha}\leq \|f\|_{\alpha}\|g\|_{\alpha}.
\end{equation}

\subsection{Operator estimates} \label{subsection:operator}

We now restrict to a specific range of values for $\alpha.$
As we have remarked above, the restrictions $L_{1}\in(0,2\pi)$ and $L_{2}\in(0,2\pi)$ imply
that $\sigma(k)>0$ for all $k\in\mathbb{Z}^{2}\setminus\{0\}.$  Clearly, then, we also have
$\sigma(k)/|k|>0,$ for all $k\in\mathbb{Z}^{2}.$  Furthermore, as $|k|\rightarrow\infty,$ we have
$\sigma(k)/|k|\rightarrow\infty.$  We introduce some notation, and we conclude the following:
\begin{equation}\nonumber
A:=\inf_{k\in\mathbb{Z}^{2}\setminus\{0\}}\frac{\sigma(k)}{|k|}>0.
\end{equation}
We require $\alpha\in(0,A).$

We introduce the linear operators $I_{1}$ and $I_{2},$ defined through their symbols as
\begin{equation}\nonumber
I_{1}h(k,t)=\frac{2\pi i k_{1}}{L_{1}}\int_{0}^{t}e^{\sigma(k)(s-t)}\hat{h}(k,s)\ ds,
\end{equation}
\begin{equation}\nonumber
I_{2}h(k,t)=\frac{2\pi i k_{2}}{L_{2}}\int_{0}^{t}e^{\sigma(k)(s-t)}\hat{h}(k,s)\ ds.
\end{equation}
Clearly these are nearly the same operator; $I_{1}$ involves a differentiation with respect to $x,$ and $I_{2}$ instead
involves a differentiation with respect to $y.$  We will show that these are bounded operators on $\mathcal{B}_{\alpha},$
but we will only include the details for $I_{1}.$

To begin, we let $h\in\mathcal{B}_{\alpha},$ and we have the definition of the norm of $I_{1}h,$
\begin{equation}\nonumber
\|I_{1}h\|_{\alpha}=\sum_{k\in\mathbb{Z}^{2}}\sup_{t\in[0,\infty)}
\left|\frac{2\pi i k_{1}}{L_{1}}e^{\alpha t |k|}
\int_{0}^{t}e^{(s-t)\sigma(k)}\hat{h}(k,s)\ ds\right|.
\end{equation}
Notice that there is no contribution to the sum when $k=0,$ since in that case we also have $k_{1}=0.$
We use this observation and the triangle inequality to arrive at
\begin{equation}\nonumber
\|I_{1}h\|_{\alpha}\leq\sum_{k\in\mathbb{Z}^{2}\setminus\{0\}}\sup_{t\in[0,\infty)}\frac{2\pi |k_{1}|}{L_{1}}e^{\alpha t |k|}
\int_{0}^{t}e^{(s-t)\sigma(k)}\left|\hat{h}(k,s)\right|\ ds.
\end{equation}
We manipulate factors of exponentials, and we rearrange using supremum inequalities:
\begin{multline}\label{normOfHTimesSomething}
\|I_{1}h\|_{\alpha}\leq\sum_{k\in\mathbb{Z}^{2}\setminus\{0\}}\sup_{t\in[0,\infty)}\frac{2\pi |k_{1}|}{L_{1}}e^{\alpha t|k|}
\int_{0}^{t}e^{(s-t)\sigma(k)-\alpha s|k|}\left|e^{\alpha s|k|}\hat{h}(k,s)\right|\ ds
\\
\leq
\left(\sum_{k\in\mathbb{Z}^{2}}\sup_{s\in[0,\infty)}\left|e^{\alpha s |k|}\hat{h}(k,s)\right|\right)
\left(\sup_{t\in[0,\infty)}\sup_{k\in\mathbb{Z}^{2}\setminus\{0\}}\frac{2\pi |k_{1}|}{L_{1}}
e^{\alpha t |k|}\int_{0}^{t}e^{(s-t)\sigma(k)-\alpha s |k|}\ ds\right)\\
=\left(\sup_{t\in[0,\infty)}\sup_{k\in\mathbb{Z}^{2}\setminus\{0\}}\frac{2\pi |k_{1}|}{L_{1}}
e^{\alpha t |k|}\int_{0}^{t}e^{(s-t)\sigma(k)-\alpha s |k|}\ ds\right)\|h\|_{\alpha}.
\end{multline}
As long as the final quantity in parentheses is finite, we have therefore demonstrated that $I_{1}$ is a bounded linear operator
on $\mathcal{B}_{\alpha}.$

Therefore, we check that this quantity is finite.  We rearrange factors and evaluate the resulting integral:
\begin{multline}\nonumber
\sup_{t\in[0,\infty)}\sup_{k\in\mathbb{Z}^{2}\setminus\{0\}}
\frac{2\pi |k_{1}|}{L_{1}}e^{\alpha t |k|}\int_{0}^{t}e^{(s-t)\sigma(k)-\alpha s |k|}\ ds
\\
=\sup_{t\in[0,\infty)}\sup_{k\in\mathbb{Z}^{2}\setminus\{0\}}\frac{2\pi |k_{1}|}{L_{1}}
e^{t(\alpha |k|-\sigma(k))}\int_{0}^{t}e^{s(\sigma(k)-\alpha |k|)}\ ds
\\
=\sup_{t\in[0,\infty)}\sup_{k\in\mathbb{Z}^{2}\setminus\{0\}}\frac{2\pi |k_{1}|}{L_{1}}\Big(e^{t(\alpha |k|-\sigma(k))}\Big)
\frac{e^{t(\sigma(k)-\alpha |k|)}-1}{\sigma(k)-\alpha |k|}.
\end{multline}
We simplify this, we bound $k_{1}$ by $k,$ and we bound $\sigma(k)/|k|$ by $A:$
\begin{multline}
\sup_{t\in[0,\infty)}\sup_{k\in\mathbb{Z}^{2}\setminus\{0\}}
\frac{2\pi |k_{1}|}{L_{1}}e^{\alpha t |k|}\int_{0}^{t}e^{(s-t)\sigma(k)-\alpha s |k|}\ ds
\leq \left(\frac{2\pi}{L_{1}}\right)\frac{1-e^{t(\alpha |k|-\sigma(k))}}{\frac{\sigma(k)}{|k|}-\alpha}.
\\
\leq \left(\frac{2\pi}{L_{1}}\right)\frac{1}{A-\alpha}.
\end{multline}
We have proven that $I_{1}$ is a bounded linear operator on $\mathcal{B}_{\alpha}.$  The same is true of $I_{2}.$

\subsection{Contraction mapping}\label{subsection:contraction}

We rewrite our operator $\mathcal{T}$ using the notations $I_{1}$ and $I_{2}:$
\begin{equation}\nonumber
\mathcal{T}(u,v)(\cdot,t)=\left(\begin{array}{cc}
e^{-t(\Delta^{2}+\Delta)}u_{0}+\frac{1}{2}I_{1}(u^{2}+v^{2})(\cdot,t)\\
e^{-t(\Delta^{2}+\Delta)}v_{0}+\frac{1}{2}I_{2}(u^{2}+v^{2})(\cdot,t)
\end{array}\right).
\end{equation}

\begin{lemma}\label{semigroupActingOnB0}
Let $f\in B_{0}$ be given, and let $\alpha\in A$ be given.  Then
$e^{-t(\Delta^{2}+\Delta)}f\in\mathcal{B}_{\alpha}.$
\end{lemma}
\begin{proof}
We estimate the norm of $e^{-t(\Delta^{2}+\Delta)}f$ as follows:
\begin{multline}\nonumber
\|e^{-t(\Delta^{2}+\Delta)}f\|_{\mathcal{B}_{\alpha}}
=\sum_{k\in\mathbb{Z}^{2}}\sup_{t\in[0,\infty)}e^{\alpha t|k|-\sigma(k)t}|\hat{f}(k)|
\\
=|\hat{f}(0)|+\sum_{k\in\mathbb{Z}^{2}\setminus\{0\}}\sup_{t\in[0,\infty)}\left(e^{\alpha-\sigma(k)/|k|}\right)^{|k|t}|\hat{f}(k)|.
\end{multline}
Since $\alpha\in(0,A),$ we see that $\alpha-\sigma(k)/|k|<0$ for all $k\in\mathbb{Z}^{2}\setminus\{0\}.$
We thus conclude $\|e^{-t(\Delta^{2}+\Delta)}f\|_{\mathcal{B}_{\alpha}}\leq |f|_{0}.$
\end{proof}

We will show that if $u_{0}$ and $v_{0}$ are sufficiently small in $B_{0},$ then  $\mathcal{T}$ is a contraction in a ball
in $\mathcal{B}_{\alpha}.$  We define the ball now, but we leave the radius to be determined.
For $r>0,$ we define $X_{r}$ to be
\begin{equation}\nonumber
X_{r}=\{(f,g)\in\mathcal{B}_{\alpha}\times\mathcal{B}_{\alpha}:
\|(f-e^{-t(\Delta^{2}+\Delta)}u_{0},g-e^{-t(\Delta^{2}+\Delta)}v_{0})\|_{\mathcal{B}_{\alpha}\times\mathcal{B}_{\alpha}}<r\}.
\end{equation}
So, this is the open ball centered at $(e^{-t(\Delta^{2}+\Delta)}u_{0},e^{-t(\Delta^{2}+\Delta)}v_{0}),$
with radius $r,$ in $\mathcal{B}_{\alpha}\times\mathcal{B}_{\alpha}.$
We let $r_{1}$ denote an upper bound on the size of $u_{0}$ and $v_{0}$ in $B_{0}:$
\begin{equation}\nonumber
|u_{0}|_{0}+|v_{0}|_{0}<r_{1}.
\end{equation}
Note that we then have a bound on the size of any element of $X_{r}:$ for all $(f,g)\in X_{r},$
\begin{equation}\label{elementOfXBound}
\|(f,g)\|_{\mathcal{B}_{\alpha}\times\mathcal{B}_{\alpha}}<r+r_{1}.
\end{equation}

We next wish to show that $\mathcal{T}$ maps $X_{r}$ to $X_{r},$ at least when $r$ and $r_{1}$ are sufficiently small.
Let $(f,g)\in X_{r}.$  We compute the distance from $\mathcal{T}(f,g)$ to the center of the ball:
\begin{equation}\nonumber
\|\mathcal{T}(f,g)-e^{-t(\Delta^{2}+\Delta)}(u_{0},v_{0})\|_{\mathcal{B}_{\alpha}\times\mathcal{B}_{\alpha}}
=\frac{1}{2}\|I_{1}(f^{2}+g^{2})\|_{\mathcal{B}_{\alpha}}+\frac{1}{2}\|I_{2}(f^{2}+g^{2})\|_{\mathcal{B}_{\alpha}}.
\end{equation}
Since we have demonstrated that $I_{1}$ and $I_{2}$ are bounded linear operators, and since $\mathcal{B}_{\alpha}$ is
a Banach algebra, this can be bounded as follows:
\begin{equation}\nonumber
\|\mathcal{T}(f,g)-e^{-t(\Delta^{2}+\Delta)}(u_{0},v_{0})\|_{\mathcal{B}_{\alpha}\times\mathcal{B}_{\alpha}}
\leq \frac{1}{2}(\|I_{1}\|+\|I_{2}\|)\left(\|f\|_{\mathcal{B}_{\alpha}}^{2}+\|g\|_{\mathcal{B}_{\alpha}}^{2}\right).
\end{equation}
Using \eqref{elementOfXBound}, we then have
\begin{equation}\label{almostDoneXToX}
\|\mathcal{T}(f,g)-e^{-t(\Delta^{2}+\Delta)}(u_{0},v_{0})\|_{\mathcal{B}_{\alpha}\times\mathcal{B}_{\alpha}}
\leq\frac{1}{2}(\|I_{1}\|+\|I_{2}\|)(r+r_{1})^{2}.
\end{equation}
We want the right-hand side of \eqref{almostDoneXToX} to be less than $r.$

Next we establish our contracting property.
Let $(f_{1},g_{1})\in X_{r}$ and $(f_{2},g_{2})\in X_{r}.$
Then, we compute the norm of the difference, after applying $\mathcal{T}:$
\begin{equation}\nonumber
\|\mathcal{T}(f_{1},g_{1})-\mathcal{T}(f_{2},g_{2})\|_{\mathcal{B}_{\alpha}\times\mathcal{B}_{\alpha}}
\leq \frac{1}{2}(\|I_{1}\|+\|I_{2}\|)\|f_{1}^{2}-f_{2}^{2}\|_{\mathcal{B}_{\alpha}}
+\frac{1}{2}(\|I_{1}\|+\|I_{2})\|g_{1}^{2}-g_{2}^{2}\|_{\mathcal{B}_{\alpha}}.
\end{equation}
We use factoring, the triangle inequality, and \eqref{elementOfXBound}, to bound this:
\begin{multline}\nonumber
\|\mathcal{T}(f_{1},g_{1})-\mathcal{T}(f_{2},g_{2})\|_{\mathcal{B}_{\alpha}\times\mathcal{B}_{\alpha}}
\\
\leq\frac{1}{2}(\|I_{1}\|+\|I_{2}\|)(\|f_{1}\|_{\mathcal{B}_{\alpha}}+\|f_{2}\|_{\mathcal{B}_{\alpha}})
\|f_{1}-f_{2}\|_{\mathcal{B}_{\alpha}}+
\\
+\frac{1}{2}(\|I_{1}\|+\|I_{2}\|)(\|g_{1}\|_{\mathcal{B}_{\alpha}}+\|g_{2}\|_{\mathcal{B}_{\alpha}})
\|g_{1}-g_{2}\|_{\mathcal{B}_{\alpha}}
\\
\leq(\|I_{1}\|+\|I_{2}\|)(r+r_{1})\|(f-f_{1},g-g_{1})\|_{\mathcal{B}_{\alpha}\times\mathcal{B}_{\alpha}}.
\end{multline}
Thus, for the contracting property, we require
\begin{equation}\label{contractingCondition}
(\|I_{1}\|+\|I_{2}\|)(r+r_{1})<1.
\end{equation}

If we take $r_{1}=\frac{1}{3(\|I_{1}\|+\|I_{2}\|)+1}$ and $r=2r_{1},$ then the right-hand side of \eqref{almostDoneXToX}
is indeed less than $r,$ and \eqref{contractingCondition} is also satisfied.  We
have proven the following theorem.

\begin{theorem}\label{firstMainTheorem}
Let $A$ be as above, and let $\alpha\in(0,A)$ be given.  Let $u_{0}\in B_{0}$ and $v_{0}\in B_{0}$ be
such that $|u_{0}|_{0}+|v_{0}|_{0}\leq\frac{1}{3(\|I_{1}\|+\|I_{2}\|)+1}.$  Let $r=\frac{2}{3(\|I_{1}\|+\|I_{2}\|)+1},$
and let $X_{r}$ be as above.  Then there exists $(u,v)\in X_{r}\subseteq\mathcal{B}_{\alpha}\times\mathcal{B}_{\alpha}$
such that $(u,v)$ is a solution of \eqref{duhamel1}, \eqref{duhamel2}.  The solution $(u,v)$ is unique in $X_{r}.$
\end{theorem}

\begin{remark}\label{firstAnalyticRemark}
Since the solution is in $\mathcal{B}_{\alpha}\times\mathcal{B}_{\alpha},$
we automatically know that the solution
exists for all $t\in[0,\infty)$ and is analytic at all
positive times, with the radius of analyticity growing linearly in time.
\end{remark}

\section{The general case: Short-time existence for small data in the Wiener algebra}

We now let the parameters $L_{1}$ and $L_{2}$ be arbitrary positive numbers; in this general case,
we are unable to prove a global existence theorem, even for small data, as in the previous section.
Instead, we let $T>0,$ and we will prove that sufficiently small solutions exist on the interval $[0,T],$
and that on this time interval, as before, the radius of analyticity of these solutions will grow linearly
in time.  In fact, this linear growth rate can be taken to be arbitrarily large; for larger values of the parameter,
$\alpha,$ measuring the growth rate of the radius of analyticity, and for larger values of the time horizon, $T,$
the amplitude of our solutions must be taken to be smaller.

The details of the proof in the current section are similar to the proof of Theorem \ref{firstMainTheorem},
so we do not reproduce every detail.  Instead, we will focus on the differences with the previous proof.
The first difference is in the definition of the function space.  Let $\alpha>0$ and $T>0$ be given.  We define
\begin{equation}\nonumber
\mathcal{B}_{\alpha,T}=\{f:\|f\|_{\alpha,T}<\infty\},
\end{equation}
where the norm is defined by
\begin{equation}\nonumber
\|f\|_{\alpha,T}=\sum_{k\in\mathbb{Z}^{2}}\sup_{t\in[0,T]}\left(e^{\alpha t |k|}|\hat{f}(k,t)|\right).
\end{equation}
Just as the space $\mathcal{B}_{\alpha}$ was a Banach algebra, so is $\mathcal{B}_{\alpha,T},$
with the estimate
\begin{equation}\nonumber
\|fg\|_{\mathcal{B}_{\alpha,T}}\leq\|f\|_{\mathcal{B}_{\alpha,T}}\|g\|_{\mathcal{B}_{\alpha,T}}.
\end{equation}

Next, we need a version of Lemma \ref{semigroupActingOnB0}.
\begin{lemma}\label{semigroupLemma2}
Let $\alpha>0$ and $T>0$ be given, and let $f\in B_{0}$ be given.  Then $e^{-t(\Delta^{2}+\Delta)}f\in\mathcal{B}_{\alpha,T}.$
\end{lemma}
\begin{proof}
We begin by writing the norm of $e^{-t(\Delta^{2}+\Delta)}f$ as follows:
\begin{multline}\label{lemmaWithTSpaces1}
\|e^{-t(\Delta^{2}+\Delta)}f\|_{\mathcal{B}_{\alpha,T}}
=\sum_{k\in\mathbb{Z}^{2}}\sup_{t\in[0,T]}e^{\alpha t|k|-\sigma(k)t}|\hat{f}(k)|
\\
=|\hat{f}(0)|+\sum_{k\in\mathbb{Z}^{2}\setminus\{0\}}\sup_{t\in[0,T]}\left(e^{\alpha-\sigma(k)/|k|}\right)^{|k|t}|\hat{f}(k)|.
\end{multline}
We decompose $\mathbb{Z}^{2}\setminus\{0\}$ as the union of the two disjoint sets $\Omega_{1}$ and $\Omega_{2},$
where these are defined as
\begin{equation}\nonumber
\Omega_{1}=\{k\in\mathbb{Z}^{2}\setminus\{0\}:\alpha-\sigma(k)/|k|\geq 0\},
\end{equation}
\begin{equation}\nonumber
\Omega_{2}=\{k\in\mathbb{Z}^{2}\setminus\{0\}:\alpha-\sigma(k)/|k|<0\}.
\end{equation}
We note that because of the nature of the symbol $\sigma,$ the set $\Omega_{1}$ is finite (possibly empty), and the
set $\Omega_{2}$ is infinite.  We thus continue our estimate by breaking the sum on the right-hand side of
\eqref{lemmaWithTSpaces1} into sums over $\Omega_{1}$ and $\Omega_{2}:$
\begin{multline}\nonumber
\|e^{-t(\Delta^{2}+\Delta)}f\|_{\mathcal{B}_{\alpha,T}}=|\hat{f}(0)|+\sum_{k\in\Omega_{1}}\sup_{t\in[0,T]}
\left(e^{\alpha-\sigma(k)/|k|}\right)^{|k|t}|\hat{f}(k)|\\
+\sum_{k\in\Omega_{2}}\sup_{t\in[0,T]}
\left(e^{\alpha-\sigma(k)/|k|}\right)^{|k|t}|\hat{f}(k)|.
\end{multline}
We manipulate each of these two sums:
\begin{multline}\nonumber
\|e^{-t(\Delta^{2}+\Delta)}f\|_{\mathcal{B}_{\alpha,T}}=|\hat{f}(0)|
+\left(\sup_{k\in\Omega_{1}}\sup_{t\in[0,T]}\left(e^{\alpha-\sigma(k)/|k|}\right)^{|k|t}\right)
\sum_{k\in\Omega_{1}}|\hat{f}(k)|\\
+\sum_{k\in\Omega_{2}}\sup_{t\in[0,T]}|\hat{f}(k)|.
\end{multline}
Since $\Omega_{1}$ is a finite set, the quantity
$\displaystyle\sup_{k\in\Omega_{1}}\sup_{t\in[0,T]}\left(e^{\alpha-\sigma(k)/|k|}\right)^{|k|t}$ is finite.
The conclusion of the lemma now follows.
\end{proof}

We next need the operator estimates for $I_{1}$ and $I_{2}.$  Since these are extremely similar operators,
we will provide details only for $I_{1}.$  We follow the argument of Section \ref{subsection:operator}
through \eqref{normOfHTimesSomething}, finding the following:
\begin{equation}\label{normOfHTimesSomething2}
\|I_{1}h\|_{\alpha,T}
\leq\left(\sup_{t\in[0,T]}\sup_{k\in\mathbb{Z}^{2}\setminus\{0\}}
\frac{2\pi|k_{1}|}{L_{1}}e^{\alpha t|k|}\int_{0}^{t}
e^{(s-t)\sigma(k)-\alpha s|k|}\ ds\right)
\|h\|_{\alpha,T}.
\end{equation}
We thus must verify that the quantity in parentheses on the right-hand side of \eqref{normOfHTimesSomething2}
is finite.  We consider the cases $k\in\Omega_{1}$ and $k\in\Omega_{2}$ separately.
Since $t\in[0,T]$ and since $\Omega_{1}$ is a finite set, we clearly have
\begin{equation}\nonumber
\left(\sup_{t\in[0,T]}\sup_{k\in\Omega_{1}}
\frac{2\pi|k_{1}|}{L_{1}}e^{\alpha t|k|}\int_{0}^{t}
e^{(s-t)\sigma(k)-\alpha s|k|}\ ds\right)<\infty.
\end{equation}
For $k\in\Omega_{2},$ we proceed instead as in Section \ref{subsection:operator}, by evaluating the integral:
\begin{multline}\nonumber
\sup_{t\in[0,T]}\sup_{k\in\Omega_{2}}
\frac{2\pi|k_{1}|}{L_{1}}e^{\alpha t|k|}\int_{0}^{t}
e^{(s-t)\sigma(k)-\alpha s|k|}\ ds
\\
=\sup_{t\in[0,T]}\sup_{k\in\Omega_{2}}
\frac{2\pi|k_{1}|}{L_{1}}e^{\alpha t|k|-t\sigma(k)}
\left(\frac{e^{t\sigma(k)-\alpha t|k|}-1}{\sigma(k)-\alpha|k|}
\right).
\end{multline}
For $k\in\Omega_{2},$ we have $\sigma(k)-\alpha|k|>0,$ so
we can bound this as
\begin{equation}\label{endOfModifiedI1Bound}
\sup_{t\in[0,T]}\sup_{k\in\Omega_{2}}
\frac{2\pi|k_{1}|}{L_{1}}e^{\alpha t|k|}\int_{0}^{t}
e^{(s-t)\sigma(k)-\alpha s|k|}\ ds
\leq\sup_{k\in\Omega_{2}}
\frac{2\pi|k_{1}|}{L_{1}}
\left(\frac{1}{\sigma(k)-\alpha|k|}
\right).
\end{equation}
From the definition of the symbol $\sigma,$ we can see now that the
quantity on the right-hand side of \eqref{endOfModifiedI1Bound} is finite.
This completes the proof that $I_{1}$ is a bounded operator from $\mathcal{B}_{\alpha,T}$ to itself.

For some $r>0,$ and for $u_{0}\in B_{0}$ and $v_{0}\in B_{0},$ we then must define the ball $\tilde{X}_{r}$ to be
\begin{equation}\nonumber
\tilde{X}_{r}=\{(f,g)\in(\mathcal{B}_{\alpha,T})^{2}:
\|(f-e^{-t(\Delta^{2}+\Delta)}u_{0},g-e^{-t(\Delta^{2}+\Delta)}v_{0})\|_{(\mathcal{B}_{\alpha,T})^{2}}<r\}.
\end{equation}
Then, the proof from Section \ref{subsection:contraction} may be carried out identically, to prove the following theorem:
\begin{theorem}Let $\alpha>0$ and $T>0$ be given.  Let $u_{0}\in B_{0}$ and $v_{0}\in B_{0}$ be such that
$|u_{0}|_{0}+|v_{0}|_{0}\leq\frac{1}{3(\|I_{1}\|+\|I_{2}\|)+1}.$  Let $r=\frac{2}{3(\|I_{1}\|+\|I_{2}\|)+1},$ and let
$\tilde{X}_{r}$ be as above.  Then there exists
$(u,v)\in\tilde{X}_{r}\subseteq\mathcal{B}_{\alpha,T}\times\mathcal{B}_{\alpha,T}$ such that $(u,v)$ is a solution of
\eqref{duhamel1}, \eqref{duhamel2}.  The solution $(u,v)$ is unique in $\tilde{X}_{r}.$
\end{theorem}

\begin{remark} \label{secondAnalyticityRemark}
As in Remark \ref{firstAnalyticRemark}, we see from the definition of the spaces $\mathcal{B}_{\alpha,T}$ that
these solutions need not be analytic initially, but become analytic at any positive time, and the radius of analyticity
grows like $\alpha t.$  Moreover, this linear growth rate, $\alpha,$ can be made arbitrarily large, with the caveat that
for larger $\alpha,$ the amplitude threshold $r$ must be taken smaller.
\end{remark}

Finally, we further remark that the proof of this section goes through with few changes to provide the analagous theorem
on the domain $\mathbb{R}^{2}.$  To state the theorem, we must introduce the appropriate function spaces.
Define the Wiener algebra on $\mathbb{R}^{2}$ to be the set of functions $f:\mathbb{R}^{2}\rightarrow\mathbb{R}$ with
integrable Fourier transform:
\begin{equation}\nonumber
\breve{B}_{0}=\left\{f:\int_{\mathbb{R}^{2}}|\hat{f}(\xi)|\ d\xi<\infty\right\}.
\end{equation}
We will repeat the previous notation for the norm in this space, denoting
\begin{equation}\nonumber
|f|_{0}=\int_{\mathbb{R}^{2}}|\hat{f}(\xi)|\ d\xi.
\end{equation}
Given $\alpha>0$ and $T>0,$ we define $\breve{\mathcal{B}}_{\alpha,T}$ in the corresponding way:
\begin{equation}\nonumber
\breve{\mathcal{B}}_{\alpha,T}=\{f:\|f\|_{\breve{\mathcal{B}}_{\alpha,T}}<\infty\},
\end{equation}
where the norm is defined as
\begin{equation}\nonumber
\|f\|_{\breve{\mathcal{B}}_{\alpha,T}}=\int_{\mathbb{R}^{2}}\sup_{t\in[0,T]}e^{\alpha t |\xi|}|\hat{f}(\xi,t)| \ d\xi.
\end{equation}
Then, the analogous lemma to Lemma \ref{semigroupLemma2} holds.
We may thus define the ball $\breve{X}_{r},$ for some $u_{0}\in\breve{B}_{0},$ $v_{0}\in\breve{B}_{0},$ and $r>0:$
\begin{equation}\nonumber
\breve{X}_{r}=\{(f,g)\in(\breve{\mathcal{B}}_{\alpha,T})^{2}:
\|(f-e^{-t(\Delta^{2}+\Delta)}u_{0},g-e^{-t(\Delta^{2}+\Delta)}v_{0})\|_{(\breve{\mathcal{B}}_{\alpha,T})^{2}}<r\}.
\end{equation}
The definitions of $I_{1}$ and $I_{2}$ need not be changed, as long as $k$ is now understood to be the continuous
Fourier variable.  The proof that $I_{1}$ and $I_{2}$ are bounded operators on $\breve{\mathcal{B}}_{\alpha,T}$
is entirely similar to the previous calculation of the present section.  Again, the proof of Section \ref{subsection:contraction}
may be carried out identically, to prove the following theorem:
\begin{theorem}Let $\alpha>0$ and $T>0$ be given.  Let $u_{0}\in \breve{B}_{0}$ and $v_{0}\in \breve{B}_{0}$ be such that
$|u_{0}|_{0}+|v_{0}|_{0}\leq\frac{1}{3(\|I_{1}\|+\|I_{2}\|)+1}.$  Let $r=\frac{2}{3(\|I_{1}\|+\|I_{2}\|)+1},$ and let
$\breve{X}_{r}$ be as above.  Then there exists
$(u,v)\in\breve{X}_{r}\subseteq\breve{\mathcal{B}}_{\alpha,T}\times\breve{\mathcal{B}}_{\alpha,T}$
such that $(u,v)$ is a solution of
\eqref{duhamel1}, \eqref{duhamel2}.  The solution $(u,v)$ is unique in $\breve{X}_{r}.$
\end{theorem}
Remark \ref{secondAnalyticityRemark} applies in this case as well.

\section{The general case: Short-time existence for large data}
\label{s:shortime}

In this section, we present a proof of short-time existence for large data in
$L^2(\TT^2)$ when growing modes are present. That is, in this section $L_1$ and
$L_2$ can take any value in $(0,\infty)$. We include a proof of this result for
completeness. In fact, short-time existence  in Gevrey spaces of the solution in
the whole space were obtained in \cite{BS07}.

We employ the same mild formulation used for data in the Wiener algebra and
again a contraction mapping argument.  Because of the presence of growing modes,
it does not follow directly from this proof that global existence holds for
sufficiently small data in $L^2(\TT^2)$.  We also choose to work with $L^2$ and
$L^2$-based Sobolev spaces  $H^s(\TT^2)=W^{s,2}(\TT^2)$, $s\in \RR$, so as to
use elementary Fourier analysis and  not to obscure the proof with technical
details, but a similar result is expected to hold in $L^p(\TT^2)$, $1<p<\infty$.
However, a Littlewood-Paley characterization of  $L^p$ and $W^{s,p}$ is needed
in this case.

Using \eqref{duhamel1} and \eqref{duhamel2}, and setting
\[
             \bu = (u,v): \TT^2 \times \RR^+ \to \RR^2,
\]
we again write the 2DKS as a fixed point equation:
\[
        \bu(t) = e^{-t\cL} \bu_0 -\int_0^t e^{-(t-\tau) \cL}
\nabla\left(\frac{|\bu|^2}{2}\right)\, d\tau
        = \cT(\bu(t)),
\]
where, for convenience we have introduced the notation $\bu(t)(x,y):=\bu(x,y,
t)$ and the operator
\[
       \cL:=\Delta^2+\Delta.
\]

Again for convenience, we set \ $\tbk :=  2\pi
\big(\frac{k_1}{L_1},\frac{k_2}{L_2}\big)$, where $\bk=(k_1,k_2)\in \ZZ^2$, so
$\tbk\in \widetilde\ZZ^2:=2\pi L_1^{-1}\ZZ\times 2\pi L_2^{-1}\ZZ$. With slight
abuse of notation, we write $\what{f}(\tbk)$ for $\what{f}(\bk)$, and  similarly
$\sigma(\tbk)$ for $\sigma(\bk)$, where $\sigma$ is given in
\eqref{symbolOfLinearPart}.

By Plancherel's formula, the norm in $H^s(\TT^2)$ can then be expressed as:
\begin{equation}\label{e:HsNorm}
    \|f\|_{H^s(\TT^2)}^2 = \sum_{\tbk\in \widetilde{\ZZ}^2} (1+|\tbk|^2)^s
|\what{f}(\tbk)|^2, \qquad
    s\in \RR
\end{equation}
and $L^2(\TT^2)\equiv H^0(\TT^2)$. Since we consider data and solutions with
finite energy, it will be convenient to work with the norm in homogeneous
Sobolev spaces $\dot{H}^s(\TT^2)$, defined by:
\begin{equation}\label{e:HsHomNorm}
    \|f\|_{\dot{H}^s(\TT^2)}^2 := \sum_{\tbk\in \widetilde{\ZZ}^2\setminus
\{0\}} |\tbk|^{2s}
    |\what{f}(\tbk)| ^2 = \||\nabla|^s f\|_{L^2(\TT^2)}^2, \qquad      s\in \RR,
\end{equation}
and observe that \ $f\in H^s \Leftrightarrow f\in \dot{H}^s$ if $f\in L^2$.
Above, $|\nabla|^s$ denotes the Fourier multiplier with symbol $|\tbk|^s$,
$\tbk\ne \mathbf{0}$. This definition agrees with the standard definition \
$\|f\|_{H^m} = \|\nabla^m f\|_{L^2}$ if $s=m\in \ZZ_+$.

Our main result in this section is the following Theorem.

\begin{theorem} \label{t:ShortTimeExt}
 Let $\bu_0\in L^2$. Then, there exists $0<T<\infty$ and a unique mild solution
$\bu$ of 2DKS
 on $[0,T)$ with initial data $\bu_0$ such that $\bu\in C([0,T);L^2)$. In
addition, for all $t\in(0,T),$ the solution satisfies $\bu(t)\in
 H^s$, for all $0\leq s<5/3,$ and also satisfies:
 \[
      \sup_{0<t<T} e^t\, t^{-\frac{s}{4}} \|\bu(t)\|_{\dot{H}^s} \leq C(L_1,
L_2, s, \|\bu_0\|_{L^2}),
      \qquad 1\leq s< 5/3.
 \]
\end{theorem}

In what follows, unless otherwise noted, $C$ denotes a generic constant that may
depend on indices, such as $s$,  and $L_1,L_2$, but not on $\bu$, $t$, or $\bk$.

\subsection{Operator estimates} We estimate the operator norm of the semigroup
$e^{-t\cL}$, $t\geq 0$,  in $H^s$ and its smoothing properties for $t>0$.  We
define the operator $e^{-t\cL}$, as before, simply by:
\[
   e^{-t\cL} f(x) = \cF^{-1}(e^{-t\sigma(\tbk)}\,\what{f}(\tbk)),
\]
whenever this expression is well defined, where $\cF$ denotes the Fourier
Transform on $\TT^2$.
We observe that $e^{-t\cL}$ is strongly continuous in any Sobolev space $H^s$.

Given $L_1, L_2$, there is a finite number of (distinct) frequencies $\tbk_i$,
$i=1,\ldots, N$,   depending on $L_1, L_2$,  for which $\sigma (\tbk)<0$. We
order them by increasing size, that is, $0<|\tbk_1|\leq
|\tbk_2|\leq\ldots\leq|\tbk_N|<1$.
In fact, $\sigma(\tbk)\leq 0$ corresponds to the parabolic region
$\kappa^2-\kappa\leq 0$, where $\kappa = |\tbk|^2$. Therefore, $\sigma(\tbk)\geq
-1/4$, and its minimum occurs at the frequency $\tbk_j$ such that $|\tbk|^2$ is
closest to $1/2$, which we will henceforth denote by $\tbk^0$ and which depends
only on $L_1$ and $L_2$.

We therefore immediately have:
\begin{equation} \label{e:SobolevBound1}
   \|e^{-t\cL}f\|_{\dot{H}^s} \leq e^{-(t\sigma(\tbk_0))} \| f\|_{H^s} \leq C\,
e^{t/4}
   \, \|f\|_{\dot{H}^s}, \qquad \forall t\geq 0, \; s\in \RR.
\end{equation}
We will need also smoothing estimates of the semigroup for $t>0$. We temporarily
fix two numbers $s,r\in \RR$, $r<s$. Using again the Plancherel formula gives:
\[
     \|e^{-t\cL} f\|_{\dot{H}^s}^2 \leq \||\tbk|^{2(s-r)} e^{-2t(|\tbk|^4-|\tbk|^2)}
     \|_{\ell^\infty(\tilde{\ZZ}^2)} \|f\|_{\dot{H}^r}^2.
\]
Therefore, it is enough to estimate the first factor on the right.
An elementary calculation gives:
\[
    \kappa^{(s-r)}\, e^{-\kappa^2 +\sqrt{t} \kappa} \leq C\, e^t \begin{cases}
     1, & \quad \kappa \geq \sqrt{t}, \\
       t^{(s-r)/2}, & \quad 0<\kappa<\sqrt{t},
    \end{cases}
\]
from which it follows, setting $\kappa=t^{1/2}|\tbk|^2$:
\[
     |\tbk|^{s-r} e^{-t(|\tbk|^4-|\tbk|^2)}
       \leq C\,  e^t \begin{cases}
        t^{(r-s)/2}, & \quad  |\tbk|\geq 1, \\
     1, & \quad 0< |\tbk| <1,
    \end{cases}\  \leq C\, e^t\ \max(1,t^{(r-s)/2}).
\]
Therefore, we obtain  the following estimate:
\begin{equation}\label{e:SobolevBound2}
   \|e^{-t\cL}f\|_{\dot{H}^s} \leq  C\, e^{t/2}\ \max(1,t^{(r-s)/4})
\|f\|_{\dot{H}^r}
\end{equation}

We will also need to estimate the action $e^{-t \cL}$ from $L^1$ to $L^2$ to
bound the non-linear term. This is more easily done in the context of the Wiener
algebra, using the algebra structure. Here, instead, we use that the Fourier
transform is well behaved in $L^1$ and $L^2$.
By Plancherel's again, and Young-Hausd\"orff inequalities, we have:
\[
  \begin{aligned}
   \|e^{-t \cL} f\|_{L^2(\TT^2)} &\leq  (\sup_{\tbk\in \tilde{\ZZ}^2}
|\widehat{f}(\tbk)|) \;
   \|e^{-t\sigma(\cdot)}\|_{\ell^2(\tilde{\ZZ}^2)}\\
    & \leq  \|f\|_{L^1(\TT^2)} \;
\|e^{-t\sigma(\cdot)}\|_{\ell^2(\tilde{\ZZ}^2)}.
  \end{aligned}
\]
So, it is enough to estimate the last term on the right. To this end, we perform
again a ``high-low" frequency decomposition, as follows:
\[
  \sum_{\tbk \in \tilde{\ZZ}^2} e^{-2t\sigma(\tbk)} \leq
  \sum_{j=1}^N e^{-2t\sigma(\tbk_j)} + \sum_{\tbk \in \tilde{\ZZ}^2, |\tbk|>
|\tbk_N| }
       e^{-2t\sigma(\tbk)}= \II_1+\II_2.
\]
The first sum can be easily estimated as before by the maximum of the symbol,
since $N$ depends only on the periods:
\[
    \II_1\leq C\, e^{t}.
\]
The second sum can be estimated as follows. First, observe that
\[
   \II_2 \leq  \sum_{\tbk \in \tilde{\ZZ}^2, |\tbk|> |\tbk_N| } e^{-2 t \beta
|\tbk|^4}, \qquad \beta=\beta(L_1,L_2)>0,
\]
since $|\tbk|> 1$ if $|\tbk|> |\tbk_N|$, by construction, so $\alpha
|\tbk|^4\geq |\tbk|^2$ for some $0<\alpha<1$, depending on the periods, hence
$\sigma(\tbk)>(1-\alpha)|\tbk|^4$. Then:
\[
   \sum_{\tbk \in \tilde{\ZZ}^2, |\tbk|> |\tbk_N| } e^{- 2 t \beta |\tbk|^4}\leq
   \int_{\RR^2} e^{- 2 t\beta \mathbf{x}^4}\, d\mathbf{x} =
\frac{\pi^{3/2}}{2\sqrt{t 2 \beta}}.
\]
Therefore, putting together these estimates, we obtain:
\begin{equation} \label{e:L1Bound}
  \|e^{-t \cL} f\|_{L^2(\TT^2)} \leq C\, e^{t/2} \max\big(1,t^{-1/4}\big)\,
\|f\|_{L^1(\TT^2)}, \qquad t>0.
\end{equation}

Combining \eqref{e:SobolevBound1}, \eqref{e:SobolevBound2}, and
\eqref{e:L1Bound}, and using the semigroup property, we finally have:
\begin{equation} \label{e:SemigroupBound}
 \begin{aligned}
   \|e^{-t \cL} f\|_{\dot{H}^s(\TT^2)} &\leq C\, e^{\frac{t}{2}}
\max\big(1,t^{-\frac{s}{4}}\big) \,
   \|e^{\frac{t}{2} \cL} f\|_{L^2(\TT^2)} \\
   &\leq C\, e^{t} \max\big(1,t^{-\frac{s+1}{4}\big)}\,
   \|f\|_{L^1(\TT^2)} , \qquad   s>0, \; t>0.
 \end{aligned}
\end{equation}
Observe that this estimate implies that  $e^{t \cL}$ as a map from $L^1$ into
$\dot{H}^s$ is locally integrable in time as long as $0<s<3$.

In the next section, we apply the operator bounds on $e^{-t\cL}$ component-wise
on $\bu$.

\subsection{The contraction mapping argument}

We introduce a Banach space adapted to the non-linear map $\cT$ obtained via the
Duhamel's formula and we will apply Banach contraction mapping to a suitable
ball in this space.

In the remainder of this section, we fix $1\leq s<5/3$, and we choose an
arbitrary initial data $\bu_0\in L^2$. recall that $\cT=\cT_{\bu_0}$ even if we
do not explicitly show this dependence.

Given $0<T\leq \infty$, , we define the space:
 \begin{multline}\label{e:AdaptedSpaceDef}
   \cX_T = \cX_{s,T} := \{ \bu: \TT^2 \times \RR_+\to \RR^2 \;\mid\;
   \bu\in C((0,T);L^2), \\ \;
   e^{-t} \,t^{\frac{s}{4}}\, \bu \in L^\infty((0,T);\dot{H}^s)\},
 \end{multline}
which is a Banach space equipped with the norm:
\[
   \|\bu\|_{\cX_T} := \sup_{0<t<T} \|\bu(t)\|_{L^2} + \sup_{0<t<T} e^{-t}
\,t^{\frac{s}{4}}
   \|\bu(t)\|_{\dot{H}^s}.
\]

We will prove the following bound:
\begin{equation*}
  \|\cT(\bu)-\cT(\bv)\|_{\cX_T} \leq C(T)\, (\|\bu\|_{\cX_T} +\|\bv\|_{\cX_T})
\, \|\bu-\bv\|_{\cX_T},
\end{equation*}
with an explicit dependence of the constant $C$ on $T$.
We observe that, thanks to  \eqref{e:SemigroupBound}, $\bu=e^{-t\cL} \bu_0\in
\cX_T$ for any $0<T<\infty$. Then, since $\cT$ is a quadratic map and $e^{-t\cL}$
is strongly continuous on $H^s$ for any $s$, establishing a bound of this type
proves that $\cT:\cX_T\to \cX_T$ and that $\cT$ is locally Lipschitz in $\cX_T$.

Since $\nabla$ and $\cL$ commute as Fourier multipliers on the torus, we write:
\[
  \|\cT(\bu)-\cT(\bv)\|_{\cX_T} = \frac{1}{2} \left\|\int_0^t  \nabla e^{-t\cL}\,
(\bu+\bv)(\tau)\cdot (\bu-\bv)
   (\tau) \, d\tau\right\|_{\cX_T}.
\]
We first bound the norm in $L^\infty((0,T);L^2)$, which is simply done using
\eqref{e:SemigroupBound} with $s=1$, Minkowski's inequality for integrals, and
H\"older's inequality:
\[
   \begin{aligned}
        \Bigg\|\int_0^t  \nabla e^{-t\cL}&\,\big[ (\bu+\bv)(\tau)\cdot (\bu-\bv)
   (\tau) \big]\, d\tau\Bigg\|_{L^2}\\
   &\leq  C\,  \int^t_0  e^{(t-\tau)} \max\big(1,(t-\tau)^{-\frac{1}{2}}\big)
\|(\bu+\bv)(\tau)\cdot (\bu-\bv)(\tau)\|_{L^1}\, d\tau \\
    & \leq  C\,   e^t\,(t^{\frac{1}{2}}+t) (\|\bu\|_{L^\infty ((0,T);L^2)}+
    \|\bv\|_{L^\infty((0,T);L^2)}) \,
    \|\bu-\bv\|_{L^\infty((0,T);L^2)}.
   \end{aligned}
\]
By using that $e^t$ and $t^{\frac{1}{2}}+t$ are strictly increasing, it follows:
\begin{equation}\label{e:L2LipEst}
      \|\cT(\bu) - \cT(\bv)\|_{L^\infty((0,T);L^2)} \leq  C\,
e^T\,(T^{\frac{1}{2}}+T)
    (\|\bu\|_{\cX_T}+ \|\bv\|_{\cX_T}) \,
    \|\bu-\bv\|_{\cX_T}.
\end{equation}

We next tackle the estimate in $\dot{H}^s$.  Thanks to  \eqref{e:HsHomNorm}, we
need to bound, for $0<t<T$:

   \begin{multline}\nonumber
       \Bigg\|\int_0^t  |\nabla|^{s}  e^{(t-\tau)\cL}\nabla\,\big[
(\bu+\bv)(\tau)\cdot (\bu-\bv)
   (\tau) \big]\, d\tau\Bigg\|_{L^2} \\
     =   \left\|\int_0^t  |\nabla|^{s} e^{\frac{(t-\tau)\cL}{2}}
     \big[ \nabla e^{\frac{(t-\tau)\cL}{2}}
     \, \big((\bu+\bv)(\tau)\cdot (\bu-\bv)(\tau)\big) \big]\right\|_{L^2}\, d\tau.
\end{multline}
We use \eqref{e:SobolevBound2}, and continue as follows:
\[
\begin{aligned}
       \Bigg\|\int_0^t&  |\nabla|^{s}  e^{(t-\tau)\cL}\nabla\,\big[
(\bu+\bv)(\tau)\cdot (\bu-\bv)
   (\tau) \big]\, d\tau\Bigg\|_{L^2} \\
     & \leq C\, \int^t_0 e^{\frac{(t-\tau)}{2}} \max \left(1,
     \frac{1}{(t-\tau)^{\frac{s}{4}}}\right)\, \times \\
     & \qquad \quad \qquad \,  \|e^{\frac{(t-\tau)\cL}{2}} \, \nabla \,\left((
     \bu+  \bv)(\tau)\cdot (\bu-\bv)(\tau)\right) \|_{L^2} \, d\tau \\
     &\leq C\, \int^t_0 e^{(t-\tau)} \max
     \left(1,\frac{1}{(t-\tau)^{\frac{s}{4}}}
      \right) \max \left(1,\frac{1}{(t-\tau)^{\frac{s-1}{4}}}\right) \times \\
     &   \qquad \quad \qquad \, \| |\nabla|^s \,\left(( \bu+  \bv)(\tau)\cdot
(\bu-\bv)(\tau)\right)
     \|_{L^2} \, d\tau \\
     &\leq C\, \int^t_0 e^{(t-\tau)} \max
\left(1,\frac{1}{(t-\tau)^{\frac{2s-1}{4}}}
      \right)\,  e^{2\tau} \,\max\left(1,\frac{1}{\tau^{\frac{s}{2}}}\right) \,
d\tau\, \times \\
       &\qquad \qquad \qquad \,   \left( \| \bu\|_{\cX_T}+\|\bv\|_{\cX_T}
\right)\,
     \|(\bu-\bv)\|_{\cX_T},
   \end{aligned}
\]
where we used that $\dot{H}^s\cap L^2$ is an algebra for $s>1$ (see e.g.
\cite{TayPDEIII}), and \eqref{e:SobolevBound2} with $r=0$.
A similar estimate can be obtained if $s=1$, using instead Leibniz formula and
\eqref{e:SemigroupBound}.
Next, we estimate the integral in $\tau$ on the last line above. To do so, we
consider different cases, depending on whether $t$ and/or $\tau$ are less or
greater than one. Combining these different cases, we have the following
estimate:
\[
  \begin{aligned}
  \int^t_0 e^{(t-\tau)} \max \Big(1,&\frac{1}{(t-\tau)^{\frac{2s-1}{4}}}\Big)
  \,  e^{2\tau} \,\max\left(1,\frac{1}{\tau^{\frac{s}{2}}}\right) \,
d\tau\, \\
    &\qquad \qquad \qquad \quad \leq C\, e^{3t} \, \left[  t+ t^{1-\frac{s}{2}}
    + t^{\frac{5}{4}-s} \right],
  \end{aligned}
\]
where $C$ can be estimated more explicitly (in terms of the Gamma function),
but it is not needed for our purposes.
Therefore:
\begin{equation} \label{e:SobolevLipEst}
  \begin{aligned}
   \sup_{0<t<T} e^{-t}\, t^{s/4}  \,\|\int_0^t  |\nabla|^{s}
e^{(t-\tau)\cL}\nabla\,\big[ (\bu+\bv)
  & (\tau)\cdot (\bu-\bv)(\tau) \big]\, d\tau\|_{L^2} \\
     \leq C\, \sup_{0<t<T} e^{2t} \, \big[ t^{1+\frac{s}{4}}+
t^{1-\frac{s}{4}} +
t^{\frac{5-3s}{4}} \big] &     (\|\bu\|_{\cX_T} +\|\bv\|_{\cX_T})b \,
\|\bu-\bv\|_{\cX_T},
   \\  \leq C\, e^{2T} \big[ T^{1+\frac{s}{4}} + T^{1-\frac{s}{4}} +
T^{\frac{5-3s}{4}} \big]\times
&     (\|\bu\|_{\cX_T} +\|\bv\|_{\cX_T})b \, \|\bu-\bv\|_{\cX_T},
   \end{aligned}
\end{equation}
using that all exponents are positive if $s<5/3$.
Combining \eqref{e:L2LipEst} and \eqref{e:SobolevLipEst}, we finally obtain:
\begin{multline} \label{e:LocalLipEst}
   \|\cT(\bu)-\cT(\bv)\|_{\cX_T} \leq C \, g(T)
(\|\bu\|_{\cX_T} +\|\bv\|_{\cX_T})b \, \|\bu-\bv\|_{\cX_T}, \\
   g(T) := e^{2T} \begin{cases} T^{1+\frac{s}{4}}, & T\geq 1, \\
T^{\frac{5-3s}{4}}, & 0<T<1.
   \end{cases}
\end{multline}

We next show that $\cT$ maps a ball in $\cX_T$ to a ball in $\cX_T$, the size of
which depends on the size of the initial data. We let $\Bar{C}$ denote the
largest among all the constants appearing in the operator estimates and the
Lipschitz estimate on $\cT$ in $\cX_T$. We stress that this constant depends
only on $L_1$, $L_2$ and $s$.

We let $\widetilde{M}=\|\bu_0\|_{L^2}$. Then, from \eqref{e:SobolevBound2} it
follows that
\[
   \|e^{-t\cL} \bu_0\|_{\cX_T} \leq \Bar{C}\,e^T \max\big(1,T^\frac{s}{4}\big)
\|\bu_0\|_{L^2}.
\]
 Then, $\cT(\mathbf{0})=e^{-t\cL} \bu_0\in B(0,M)\subset \cX_T$ if:
\begin{equation} \label{e:BallSize}
   M> \Bar{C}\, h(T)\,\widetilde{M},
\end{equation}
where $h(T):=e^T\, \max\big(1,T^\frac{s}{4}\big)$.
Assume now that $\bu\in B(0,M)$, where $M$ satisfies this bound.
Choosing $\bv=\mathbf{0}$ in \eqref{e:LocalLipEst}, we have:
\[
  \begin{aligned}
   \|\cT(\bu)\|_{\cX_T}&\leq \|\cT(\bu)-\cT(\mathbf{0})\|_{\cX_T} +
\|\cT(\mathbf{0})
   \|_{\cX_T}\\
    &\leq C(T) \|\bu\|^2_{\cX_T} + \|e^{-t\cL} \bu_0\|_{\cX_T} \\
    &\leq \Bar{C}\, (g(T)\, M^2 + h(T)\,\widetilde{M})
  \end{aligned}
\]
It is easy to see then that $\cT(\bu) \in B(0,M)$ if $T$ is small enough; in
fact, more precisely if:
\[
     \Delta =:1-4 \Bar{C}^2 \widetilde{M}\, h(T) g(T)> 0,
\]
and $M$ is taken in the range:
\[
    0< \frac{1-\sqrt{\Delta}}{2\bar{C} g(T)} < M <
\frac{1+\sqrt{\Delta}}{2\bar{C} g(T)},
\]
which automatically gives \eqref{e:BallSize}. This condition also implies that
$\cT$ is a contraction in $B(0,M)$, since $M$ satisfies $M>\bar{C}\, g(T)
M^2$.
When $0<T\leq 1$, the condition on $\Delta$ reduces to:
\begin{equation} \label{e:ShortTimeCond}
    T< \left( \frac{1}{4 \Bar{C} \widetilde{M}}\right)^{\frac{4}{5-3s}},
\end{equation}
with the familiar inverse dependence of the time of existence on the size of the
initial data.

Applying Banach Contraction Mapping Theorem yields then a unique fixed point of
the map $\cT$ in $B(0,M)$. This fixed point is in fact the only fixed point in
$\cX_T$, since if there is another fixed point $\bu'$ in $B(0,M)'$, $M'>M$, then
$\bu=\bu'$ as $B(0,M)\subset B(0,M')$.
The proof of Theorem  \ref{t:ShortTimeExt} is complete.

Since $\bu\in H^s$, $1\leq s < 5/3$, on any interval of the form $[\delta,T)$,
$\delta>0$, we can bootstrap the regularity and conclude that for a short time
$\bu\in H^r$, $\forall r>0$, but the time of existence in $H^s$ may become
progressively shorter. Indeed, since $H^s$ is an algebra, one can repeat the
proof of Theorem \ref{t:ShortTimeExt} starting with initial data in $H^s$ to
gain regularity for $t>0$. In fact, existence and uniqueness can be more
directly obtained by ODE methods in Banach spaces if $\bu_0\in H^r$, $r>2$,
since then the non-linear term in 2DKS is bounded in $H^r$ for $\bu\in H^r$.

\section{No growing modes: global existence for small data in $L^2$}
\label{s:L2global}

In this section, we give another existence proof of global-in-time existence of
a mild solution for small data when there are no growing modes. The data is
taken small in $L^2(\TT^2)$ and with zero average. This last condition is
needed to ensure the validity of Poincar\'e's inequality on the torus and
ensures that the $L^2$ norm of the solution decays in time.

Global existence in $L^2$ for small data complements the result in the Wiener
algebra. While it is true that $\mathcal{B}_0\subset L^2$ on the torus, no
zero-average condition on the initial data is needed in $\mathcal{B}_0$, and
the proof yields a radius of analyticity that grows linearly for all time, while for
$L^2$ data we can only establish initial growth of order $t^{1/4}$ (see Section
\ref{s:analytic} below).

We reinstate the hypothesis that the periods $L_1, L_2\in (0,2\pi)$ to avoid
the existence of growing modes for the linear part of the equation. Under this
condition, following the notation of Section \ref{s:shortime},
$|\Tilde{\bk}|>1$ and $\inf_{\Tilde\bk \in\Tilde\ZZ^2\setminus \{0\}}
\sigma(\Tilde\bk)>0$. Hence the operator estimates
on $e^{-t\cL}$ are modified as follows:
\begin{align}
   &\|e^{-t\cL}f\|_{H^s}\leq
   \, \|f\|_{H^s}, \qquad  t\geq 0, \; s\in \RR,
   \label{e:SobolevBound1ngm} \\
      &\|e^{-t\cL}f\|_{\dot{H}^s} \leq  C\,t^{(r-s)/4}
    \|f\|_{\dot{H}^r},  \qquad  t> 0, \; s\geq r,
   \label{e:SobolevBound2ngm} \\
      &\|e^{-t \cL} f\|_{\dot{H}^s(\TT^2)} \leq C\, t^{-\frac{s+1}{4}}\,
     \|f\|_{L^1(\TT^2)} , \qquad  t>0, \; s>0,  \label{e:SemigroupBoundngm}
\end{align}
where $C$ depends on the periods, $s$ and $r$, but not on $t$ nor on $f$.

We will prove existence of a mild solution using an adapted space and again the
Contraction Mapping Theorem for small enough initial data with zero average.
This condition is preserved under the forward evolution in 2DKS, at least for
strong solutions.

\begin{lemma} \label{l:ZeroAverage}
 Let $\bu_0\in L^2(\TT^2)$, and let $\bu$ be a strong solution of 2DKS on
$(0,T)$ with initial data $\bu_0$. If $\bu_0$ has average zero over the torus,
then $\bu(t)$ has average zero over the torus for all $t>0$.
\end{lemma}

\begin{proof}
Since $\bu$ is a strong solution the equation is satisfied pointwise and all
terms are integrable in space over the torus $\TT^2$ and in time over $(0,T)$.
Consequently:
\[
  \frac{d}{dt} \int_{\TT^2}  \bu(t)\, d\bx = \int_{\TT^2} \pa_t \bu(t)\, d\bx  =
- \frac{1}{2} \int_{\TT^2} \nabla |\bu|^2(t)\,d\bx,
\]
where we used that, by the divergence theorem and periodicity,
\[
     \int_{\TT^2} \Delta^2 \bu(t)\, d\bx = \int_{\TT^2} \Delta \bu(t)
      \,d\bx =\mathbf{0}, \qquad \text{a.e. } t\in (0,T).
\]
Next, we note that, by periodicity again, integrating in each variable
separately, $\int_0^{L_2} \int_0^{L_1}\pa_x |u|^2(x,y,t)\,dx\,dy=0$ and
similarly for the $y$-derivative so that:
\[
         \frac{d}{dt} \int_{\TT^2}  \bu(t)\, d\bx =0.
\]
Above we have used that $\bu$ is a strong solution and hence it is continuous
on $\TT^2.$
\end{proof}

For notational convenience we will denote the subspace of functions in $L^2$
with zero average as:
\[
        \rL^2(\TT^2):= \left\{ f\in L^2(\TT^2)\;\mid\; \fint_{\TT^2} f\,
        d\bx=0\right\}
\]
By Poincar\'e's inequality, if a function $f$ has average zero over the torus,
then $\|f\|_{L^2} \leq C\, \|\nabla f\|_{L^2}$ so that $\Dot{H}^1\subset \rL^2.$

We next introduce the adapted space for the contraction mapping:
\begin{equation} \label{e:GlobalAdaptedSpeceDef}
   X_\infty := \{ \bu: \TT^2 \times [0,\infty)\to \RR^2 \; \mid\; \bu \in
L^\infty([0,+\infty);\rL^2), \sup_{0<t<\infty} t^{1/4}\,
\|\bu\|_{\Dot{H}^1} <\infty\},
\end{equation}
with norm
\[
    \|\bu\|_{X_\infty} := \max\big( \sup_{0<t<\infty} \|\bu(t)\|_{L^2},
\sup_{0<t<\infty} t^{1/4}\, \|\bu\|_{\Dot{H}^1}).
\]
By Lemma \ref{l:ZeroAverage}, $\cT(\bu)$ has average zero if $\bu$ does, where
$\cT_{\buzero}=\cT$ is again the non-linear map in Duhamel's representation.

We will also need the following elementary result.

\begin{lemma} \label{l:integral}
   Let $\alpha$, $\beta$, $\gamma$ be given non-negative numbers. If $\alpha<1$ and
$\alpha+\beta=1$, and $0<\beta+\gamma<1$, then there exists a positive
constant $C$ such that
\[
     \int_0^t \frac{1}{(t-\tau)^\alpha} \frac{1}{\tau^\beta}
     \frac{t^\gamma}{\tau^\gamma}\, d\tau < C,
\]
where $C$ may depend on $\alpha$, $\beta$, $\gamma$, but is independent of
$t\in [0,\infty)$.
\end{lemma}

The Lemma is easilty proved by making the change of variable $\tau/t=\theta$.

\begin{theorem} \label{t:L2GlobalSmallData}
  Let $L_1,L_2\in (0,2\pi)$. There exists $\epsilon>0$ small enough such that,
if \, $\buzero\in \rL^2(\TT^2)$ and $\|\buzero\|_{L^2} \leq \epsilon$, then
the initial-value problem for 2DKS with initial data $\buzero$ has a
unique solution $\bu$ in $X_\infty$.
\end{theorem}

\begin{proof}
 Let $\buzero\in \rL^2(\TT^2)$ be fixed. We first show that $\cT: X_\infty
\to X_\infty$ continuously. Throughout the proof, we employ the
standard shorthand  notation  $\,\lesssim\,$ to denote $\,\leq C\,$ with $C>0$
that is independent of $\buzero$, $\bu$, and $t$.

We decompose the map $\cT$ in its linear part, $e^{-t\cL} \buzero$, which we
call the trend as it is dominant for $t$ small for a mild solution, and the
non-linear part, $\int_0^t e^{-(t-\tau)\cL} \nabla (|\bu(\tau)|^2/2)\,d\tau$,
which we call the fluctuation.

From \eqref{e:SobolevBound1ngm}-\eqref{e:SobolevBound2ngm}, it follows
immediately that the trend belongs to $X_\infty$. We next bound the fluctuation.
We begin by considering the $L^2$ norm:
 \begin{align}
  \Bigg\| \int_0^t e^{-(t-\tau)\cL} &\nabla (|\bu(\tau)|^2/2)\,d\tau\Bigg\|_{L^2}
  \lesssim \int_0^t \frac{1}{(t-\tau)^{1/2}} \||\bu|^2(\tau))\|_{L^1} \, d\tau
    \nonumber \\
  &\lesssim \int_0^t \frac{1}{(t-\tau)^{1/2}} \|\bu(\tau))\|^2_{L^2} \, d\tau
   \nonumber \\
  &\lesssim \int_0^t \frac{1}{(t-\tau)^{1/2}} \|\bu(\tau))\|^2_{\Dot{H}^1}\,
    d\tau \nonumber\\
  & \lesssim \left(\int_0^t \frac{1}{(t-\tau)^{1/2}} \frac{1}{\tau^{1/4}}
     \frac{1}{\tau^{1/4}} \,d\tau\right) \|\bu\|^2_{X_\infty}
    \lesssim \|\bu\|^2_{X_\infty}, \label{e:L2FluctuationBound}
 \end{align}
where we have used \eqref{e:SemigroupBoundngm} with $s=1$, Poincar\'e's
inequality, and Lemma \ref{l:integral} with $\alpha=1/2$, $\beta=1/2$, and
$\gamma=0$.
We now estimate the $\Dot{H}^1$ norm of the fluctuation, that is the $L^2$ norm
of the gradient, in a similar fashion:
\begin{align}
   \Bigg\| \int_0^t e^{-(t-\tau)\cL} &\nabla^2 (|\bu(\tau)|^2/2)\,d\tau\Bigg\|_{L^2}
  \lesssim \int_0^t \frac{1}{(t-\tau)^{3/4}} \||\bu|^2(\tau))\|_{L^1} \, d\tau
    \nonumber \\
  &\lesssim \int_0^t \frac{1}{(t-\tau)^{3/4}} \|\bu(\tau))\|^2_{L^2} \, d\tau
   \nonumber \\
  &\lesssim \int_0^t \frac{1}{(t-\tau)^{3/4}} \|\bu(\tau))\|^2_{\Dot{H}^1}\,
    d\tau \nonumber\\
  & \lesssim \, \frac{1}{t^{1/4}} \left(\int_0^t \frac{1}{(t-\tau)^{3/4}}
    \frac{1}{\tau^{1/4}}  \frac{t^{1/4}}{\tau^{1/4}} \,d\tau\right)
   \|\bu\|^2_{X_\infty} \nonumber \\
  & \lesssim  \, \frac{1}{t^{1/4}}\|\bu\|^2_{X_\infty},
      \label{e:H1FluctuationBound}
\end{align}
where again we have used \eqref{e:SemigroupBoundngm} with $s=2$, Poincar\'e's
inequality, and Lemma \ref{l:integral} with $\alpha=3/4$, $\beta=1/4$, and
$\gamma=1/4$.

Combining \eqref{e:L2FluctuationBound} with \eqref{e:H1FluctuationBound} yields:
\begin{equation} \label{e:XinfinityMapBound}
    \|\cT(\bu)\|_{X^\infty} \leq A\, (\|\buzero\|_{X_\infty} +
    \|\bu\|^2_{X_\infty}),
\end{equation}
for some constant $A>0$, depending on $L_1$, and $L_2$ only.
This bound also implies a Lipschitz estimate on $\cT$ in $X_\infty$:
\begin{align}
  \|\cT(\bu)-&\cT(\bv)\|_{X_\infty}  =
    \frac{1}{2} \, \left\|  \int_0^t e^{-(t-\tau)\cL}  \big[\nabla (|\bu(\tau)|^2)
-\nabla     (|\bv(\tau)|^2)\big] \,d\tau \right\|_{X_\infty} \nonumber \\
   & =  \frac{1}{2} \, \left\|  \int_0^t e^{-(t-\tau)\cL}   \nabla \big[ \bu(\tau))
     \cdot  (\bu(\tau)- \bv(\tau)) + \bv(\tau) \cdot (\bu(\tau)-
    \bv(\tau))\big] \,d\tau \right\|_{X_\infty} \nonumber \\
   & \leq A\, (\|\bu\|_{X_\infty} + \|\bv\|_{X_\infty} )
     \|\bu-\bv\|_{X_\infty}, \label{e:LipschitzXinfinityBound}
\end{align}
where we have proceeded as in
\eqref{e:L2FluctuationBound}-\eqref{e:H1FluctuationBound}
for the last inequality and $A$ is the constant in \eqref{e:XinfinityMapBound}.

We set $\Tilde{M}:= \|\buzero\|_{L^2}$, and consider $\cT$ as a map on the ball
$B(0,M)\subset X_\infty$ with $M$ to be determined later. From the estimates
above on $\cT$, we have that $\cT: B(0,M)\to B(0,M)$ if $C \Tilde{M} + A
M^2<M$, which can be arranged by choosing, for instance, $ M= 2 C \Tilde{M}$
and $\Tilde{M} <\frac{1}{4 A C}$. Under this condition on $\Tilde{M}$, $\cT$ is
also automatically a contraction on $B(0,M)$, since
\[
    \|\cT(\bu)-\cT(\bv)\|_{X_\infty} \leq 2A M \,\|\bu-\bv\|_{X_\infty},
    \qquad \bu,\,\bv\in B(0,M).
\]
Then, by the Contraction Mapping Theorem, there is a unique fixed point $\bu$
of the map $\cT$ in $B(0,M)$. by a standard continuation argument, the solution
is unique in $X_\infty$.
\end{proof}


\section{The general case: The radius of analyticity} \label{s:analytic}

In this section, we study the analyticity of mild solutions of 2DKS for $t>0$,
and obtain an lower bound for the radius of analyticity as a function of time
and the $L^2$-norm of the initial data.

We recall from Section \ref{s:wiener} that the mild solution with initial data
in the Wiener algebra extends as an analytic function on a strip of width
growing linearly in time (Remark \ref{firstAnalyticRemark}).

For the general case, we follow the approach in \cite{GK98}, which treats the
1DKS equation in $\RR$ with $L^\infty$ data and the Navier-Stokes equations.
Our proof is very similar, except that we do not restrict to short time and
take the exponential growth of the linear part into consideration.
As remarked in \cite{GK98}, this approach lends itself well to study the
dependence on other $L^p$ norms, but we do not pursue this extension here.

Analyticity for the 2DKS has also been studied using Gevrey classes  \cite{BS07}
and growth of higher Sobolev norms \cite{SSArXiv07} (see also \cite{IS16} for a
spectral approach to related models).

We henceforth fix an initial data $\mathbf{u}_0\in L^2(\TT^2)$, with average
zero, where again no assumptions are made on the periods $L_1$ and $L_2$, and a
time $0<T<\infty$ to be determined later. We will denote the point $(x,y)\in
\RR^2$ or $\TT^2$ also by $\bx$ for notational convenience

The main result of this section is
the following theorem.

\begin{theorem} \label{t:analytic}
  Let $\bu_0\in L^2(\TT^2)$ with zero average. There exists a constant $C>0$,
independent of
$\bu_0$ such that, if $0<T<\infty$ satisfies
\[
     e^T\,\max\big( T, T^{1/2}\big) <\frac{1}{2C\,\|\bu_0\|_{L^2(\TT^2)}},
\]
then there exists a unique mild solution $\bu$ of 2DKS on $[0,T)$ with initial
data $\bu_0$ such that $\bu$ at time $t\in (0,T)$ extends as an analytic
function in the strip
\[
     \mathcal{D}_t : = \{ (\bx,\by)\in \CC^2; \; 2C|\by|< e^{-t}
\min(1,t^{1/4})\},
\]
and satisfies the bound
\[
     \|\bu(\cdot,\by,t)\|_{L^2(\TT^2)}\leq C\,\|\bu_0\|_{L^2(\TT^2)}, \quad
(\bx,\by)\in
\mathcal{D}_t,
\]
for all $0<t<T$.
\end{theorem}

We begin by constructing a suitable regular Picard iteration that will be shown
to converge to a mild solution of 2DKS. To this effect, for $n \in \NN$, let
$u^{(n)}$ denote the unique (strong) solution of the following problem:
\begin{equation} \label{e:uneq}
 \begin{cases}
   \pa_t \un +\cL \un = \nabla \big (\frac{ | \mathbf{u}^{(n-1)}|^2}{2}\big), &
\bx\in \TT^2, \; 0<t<T; \\
   \un\lfloor_{t=0}= \buzero,  & \bx\in \TT^2,
  \end{cases}
\end{equation}
where $\un$ is built recursively from $\bu^{(0)} \equiv 0$ and $\cL$ is again
the operator $\Delta^2 +\Delta$.

This solution exists in $L^\infty([0,T); L^2(\TT^2))$ for any $T$ by standard
ODE theory in Banach spaces, and it is an analytic function for $0<t<T$, given
that $\cL$ generates an analytic semigroup (not of contractions) in $L^2(\TT^2)$
on functions with average zero. This can be inferred from the smoothing
properties of $e^{-t\cL}$ or directly from the symbol of the operator
$e^{-t\cL}$ for $t>0$.
In particular, $\un$ can be extended to an entire function $\un(\bx,\by,t) + i\,
\vn(\bx,\by,t)$ on
$\RR^4\equiv \CC^2$ that is periodic in $\bx$ for each $0<t<T$. This function
therefore satisfies an analog equation  to \eqref{e:uneq}. Writing separately
the equation for its real and imaginary parts, we obtain a hierarchy of coupled
linear systems for $\un$ and $\vn$:
\begin{equation} \label{e:complexuneq}
  \begin{aligned}
      \pa_t \un +\cL \un &= \nabla \big (\frac{ | \bu^{(n-1)}|^2}{2}\big) +
\nabla \big
      (\frac{ | \bv^{(n-1)}|^2}{2}\big), \\
       \pa_t \vn +\cL \vn &=  -  \nabla \big
      (\bu^{(n-1)} \cdot \bv^{(n-1)}\big),
  \end{aligned}
\end{equation}
supplemented by the initial conditions  \ $\un(\bx,\by,0) =\bu_0(\bx)$ and
$\vn(\bx,\by,0)=0$.

We will show that the sequence $\{\un + i\, \vn\}$ of entire functions so
constructed is uniformly bounded in $L^\infty([0,T),L^2(\TT^2))$ provided $T$ is
sufficiently small, and deduce from this uniform bound that this sequence and
all its derivatives form a Cauchy sequence in $L^\infty([0,T),L^2(\TT^2))$,
therefore converging to an analytic function, which must be a classical solution
of 2DKS.

To determine the strip in $\CC^2$ where these uniform bounds hold, we introduce
the auxiliary functions:
\[
  \Un (\bx,t) := \un(\bx, \balpha t, t), \qquad  \Vn (\bx,t) :=
   \vn(\bx, \balpha t, t),
\]
where $\balpha\in \RR^2$ is a vector parameter, $\bx\in \TT^2$, $0\leq t<T$.
Note that $\Un(\bx,\mathbf{0},t)=\un(\bx,t)$ and $\Vn(\bx,\mathbf{0},t)=0$.
From the Cauchy-Riemann equations, which are satisfied coordinate-wise by the
functions $\un + i\, \vn$, one obtains the following coupled system for $\Un$
and $\Vn$:
\begin{equation} \label{e:Uneq}
 \begin{aligned}
      \pa_t \Un +\cL \Un &= - \balpha \cdot \nabla \Vn + \nabla \big (\frac{ |
\bU^{(n-1)}|^2}{2}\big) +  \nabla \big
      (\frac{ | \bV^{(n-1)}|^2}{2}\big), \\
       \pa_t \Vn +\cL \Vn &=   \balpha \cdot \nabla \Un - \nabla
       \big  (\bU^{(n-1)} \cdot \bV^{(n-1)}\big).
  \end{aligned}
\end{equation}
The solution of this system can obtained via Duhamel's formula as the fixed
point of the following integral equation:
\begin{equation} \label{e:UneqInt}
 \begin{aligned}
    \Un(t) &= e^{-t\cL} \bu_0  -\balpha \cdot \int_0^t e^{(t-\tau) \cL} \nabla
\Vn(\tau)\, d\tau+
    \int_0^t e^{(t-\tau) \cL} \big[\nabla \big (\frac{ | \bU^{(n-1)}|^2}{2}\big)
+ \\
     & \qquad \qquad \qquad +  \nabla \big
    (\frac{ | \bV^{(n-1)}|^2}{2}\big)(\tau)\big ]\, d\tau, \\
    \Vn(t) &= \balpha \cdot \int_0^t e^{(t-\tau) \cL} \nabla \Un(\tau)\, d\tau -
    \int_0^t e^{(t-\tau) \cL} \big[  \nabla \big
      (\bU^{(n-1)} \cdot \bV^{(n-1)})(\tau) \big]\, d\tau.
  \end{aligned}
\end{equation}
We will employ the above system and the operator estimates derived in Section
\ref{s:shortime} to obtain uniform bounds in $n$ on $\Un$ and $\Vn$. More
specifically, we utilize \eqref{e:SobolevBound2} with $s=1$ and $r=0$ and
\eqref{e:SemigroupBound} with $s=1$ to derive the following:
\begin{align}
  \|\Un&\|_{L^\infty([0,T);L^2(\TT^2))} \leq C\, e^T\, \big(
\|\bu_0\|_{L^2(\TT^2)} +
  | \balpha|\, (T+T^{3/4}) \, \|\Vn\|_{L^\infty([0,T);L^2(\TT^2))} \nonumber \\
   &  + (T+T^{1/2}) \,
    \|(\bU^{(n-1)})^2\|_{L^\infty([0,T);L^2(\TT^2))} \|(\bV^{(n-1)})^2\|
    _{L^\infty([0,T);L^2(\TT^2))}\big), \nonumber\\
   \|\Vn&\|_{L^\infty([0,T);L^2(\TT^2))} \leq C\, e^T\, \big(
\|\bu_0\|_{L^2(\TT^2)} +
  | \balpha|\, (T+T^{3/4}) \,  \|\Un\|_{L^\infty([0,T);L^2(\TT^2))}\nonumber \\
   & + (T+T^{1/2}) \,
    \|\bU^{(n-1)}\|_{L^\infty([0,T);L^2(\TT^2))}
\|\bV^{(n-1)}\|_{L^\infty([0,T);L^2(\TT^2))}\big).\label{e:UnboundFirst}
\end{align}

As in Section \ref{s:shortime}, we let $\tilde{M}:= \|\bu_0\|_{L^2(\TT^2)}$.

Let $g(T)$ again be the function introduced in \eqref{e:LocalLipEst}, where we
take $s=1$.
Assume now that, given $T$, $\balpha$ satisfies
\begin{equation} \label{e:FirstCond}
    |\balpha|\leq  \frac{1}{2Cg(T)}.
\end{equation}
Then we can absorb  all terms at level $n$ on the left-hand side in
\eqref{e:UnboundFirst}, giving:
\[
  \begin{aligned}
   \|\Un&\|_{L^\infty([0,T);L^2(\TT^2))}  + \|\Vn\|_{L^\infty([0,T);L^2(\TT^2))}
    \leq C\, e^T\, \big[\|\bu_0\|_{L^2(\TT^2)}    \\
    &  +   (T+T^{1/2})\big(
    \|(\bU^{(n-1)})\|_{L^\infty([0,T);L^2(\TT^2))} + \|(\bV^{(n-1)})\|
    _{L^\infty([0,T);L^2(\TT^2))}\big)^2 \big].
  \end{aligned}
\]
By induction on $n$, it follows that
\begin{equation*} 
   \|\Un\|_{L^\infty([0,T);L^2(\TT^2))}  + \|\Vn\|_{L^\infty([0,T);L^2(\TT^2))}
    \leq C\,\tilde{M},
\end{equation*}
provided $T$ satisfies
\begin{equation} \label{e:SecondCond}
   0<T\leq  \tilde{g}^{-1}\big(\frac{1}{2C\tilde{M}}\big),
\end{equation}
where the function $\tilde{g}$ is given by:
\[
     \tilde{g}(t)= e^t \begin{cases}  t, & t\geq 1, \\
                        t^{1/2}, & 0<t<1.
                       \end{cases}
\]
Setting $\by=\balpha t$, $0\leq t\leq T$, gives the following {\em a priori}
bound on the complex-valued solution.

\begin{lemma}
  Let $0<T<\infty$ satisfy condition \eqref{e:SecondCond} and let
\[
    |\by|\leq  \frac{T}{2 C\,g(T)}.
\]
Then, for all $n\in \NN$:
\begin{equation} \label{e:UVunifbound}
     \|\un(\cdot,\by,t)\|_{L^\infty([0,T);L^2(\TT^2))}  +
\|\vn(\cdot,\by,t)\|_{L^\infty([0,T);L^2(\TT^2))}
    \leq C\,\tilde{M}.
\end{equation}

\end{lemma}

From the Lemma it follows that there exists $\Bar\bu\in L^2(\TT^2)$ such
that $\un(\cdot, \mathbf{0}, t)$ converges to $\Bar\bu$ in $L^2(\TT^2)$
 for $0\leq t\leq T$ and $\Bar\bu$ is a mild solution of 2DKS. In fact, using
\eqref{e:UVunifbound} in
\eqref{e:complexuneq} with $\balpha=\mathbf{0}$ gives:
\[
 \begin{aligned}
   \|\un - \bu^{(n-1)}&\|_{L^\infty([0,T);L^2(\TT^2))}  \leq
   C\, \tilde{g}(T) \big(  \|\bu^{(n-1)}\|_{L^\infty([0,T);L^2(\TT^2))} + \\
   \qquad &
    \|\bu^{(n-2)}\|_{L^\infty([0,T);L^2(\TT^2))} \big) \,
    \|\bu^{(n-1)} - \bu^{(n-2)}\|_{L^\infty([0,T);L^2(\TT^2))}\\
   \qquad & \leq  \tilde{C}\,  \|\bu^{(n-1)} -
\bu^{(n-2)}\|_{L^\infty([0,T);L^2(\TT^2))},
 \end{aligned}
\]
with $0<\tilde{C}<1$, and then we can apply the Contraction Mapping theorem.

Next we show that $\Bar\bu$ is in fact a classical solution. To do so, following
\cite{GK98}, we call $\mathcal{D}$ the set of all points $(\bx,\by,t)$, such
that $t \in (0,T)$, where $T$ satisfies \eqref{e:SecondCond}, $\bx\in \RR^2$,
and
$|\by|=|\balpha| t$ with $\balpha$ satisfying \eqref{e:FirstCond}. We observe
that, for fixed $t\in (0,T)$, the set $\mathcal{D}_t:= \{(\bx,\by)\in \CC^2; \;
(\bx,\by,t)\in \mathcal{D}\}$
is open in $\CC^2$.
Then, one derives from \eqref{e:UVunifbound} the following estimate on the
sequence $\{\un\}$:
\begin{equation}
   \sup_{0<t<T} \int_{|\by|<|\balpha| t} \int_{\TT^2} |\bu^{(n)}(\bx,\by,t)|^2\,
dy\,dy \leq M', \qquad \forall n\in \NN,
\end{equation}
for some constant $M'$, which depends on $T$, but not on $n$. Since $\un$ is
holomorphic in $(\bx,\by)\in \mathcal{D}_t$ for each $0<t<T$, the family
$\{\un\}$ is locally uniformly bounded there, and hence it is a normal family by
Montel's Theorem in several complex variables. This means that every derivative
of $\un$ in $\bx$ and $\by$ are also uniformly bounded in $n$  on compact
subsets of $\mathcal{D}$. Then, from \eqref{e:complexuneq}, we also conclude
that all time deriatives of $\un$ are uniformly bounded in $n$ on the same sets.
Therefore, by a diagonal argument, there exists $\bu$, $\bv\in
C^\infty(\mathcal{D})$ and a subsequence $\{\bu^{(n_k)},\, \bv^{(n_k)})\}$ such
that
$$
 \begin{aligned}
   \pa^\gamma_t \pa^\beta_\bx
   \bu^{(n_k)}&\underset{k\to\infty}{\longrightarrow} \bu,\\
   \pa^\gamma_t \pa^\beta_\bx
   \bv^{(n_k)}&\underset{k\to\infty}{\longrightarrow} \bv,
 \end{aligned}
$$
for every index $\gamma$ and multi-index $\beta$, uniformly on compact subsets
of $\mathcal{D}$, by a diagonal argument. We remark that since $\un$ is
periodic in $\bx$, $\bu$ is periodic as well and, hence it can be identified
with a function on $\TT^2$ for fixed $\by,t$. Furthermore, $\bu$ satisfies the
same bound as in \eqref{e:UVunifbound} by lower semicontinuity.

By uniqueness of the limit, then, $\Bar\bu\equiv \bu$ in
$C^\infty(\mathcal{D})$, and hence $\Bar\bu$ is a classical solution of 2DKS
on $(0,T)$ in view of the ``weak=strong'' uniqueness, which holds for the
equation by standard results on mild solutions (see e.g. \cite{Pazy}). Lastly,
$\Bar\bu$ is analytic in $\mathcal{D}$, and in fact
$\Bar\bu(\bx,\by,t)) +i \bv(\bx,\by,t)$ is the analytic extension of
$\Bar\bu(\bx,\by,t)$ to
the strip $|\by|<|\balpha| t$ in $\CC^2$.

As already observed in \cite{SSArXiv07}, the size of the $L^2$ norm controls
the growth of higher Sobolev norm leading to a continuation/blow up criterion
for 2DKS.

\begin{proposition} \label{p:blowup}
 Let $0<T<\infty$ and let $\bu$ be a mild solution of 2DKS with initial data
$\bu_0\in L^2$ on $[0,T]$. Then, if
\begin{equation} \label{e:blowup}
    \limsup_{t\to T_-} \|\bu(t)\|_{L^2(\TT^2)} < \infty,
\end{equation}
the solution can be continued to $[0,T']$ for some $T'>T$ and it is a classical
solution on $(0,T')$.
\end{proposition}

\begin{proof}
Since \eqref{e:blowup} holds, there exists a constant $0<M<\infty$ such that
for any $0<t<T$, $ \|\bu(t)\|_{L^2(\TT^2)}<M$. Fix an arbitrary $\delta>0$
small  and consider $0<t_0<1$ to be a time that satisfies
\[
   t_0<\frac{1}{2CM}^2.
\]
Note that $t_0$ is independent of $\delta$.
Then, by Theorem \ref{t:analytic}, there exists a mild solution
$\Tilde\bu(\tau)$ to 2DKS on $[0,t_0)$ with initial data $\bu(T-\delta)$, which
is classical on $(0,t_0)$. The following
\[
      \Bar{u}(t):=\begin{cases}
                      \Bar\bu(t)\equiv \bu(t), & 0\leq t<T-\delta, \\
                      \Bar\bu(t) \equiv \Tilde\bu(\tau), & t=(T-\delta)+\tau,
\;                       0\leq\tau<t_0,
                  \end{cases}
\]
is a mild solution to 2DKS with initial data $\bu_0$ and, hence, by uniqueness
of mild solutions, $\Bar\bu$ coincides with $\bu$ on $[0,T-\delta)$. It is now
enough to take $t_0>\delta$ to conclude.
\end{proof}

%



\end{document}